\documentclass{article}
\usepackage{amsmath}
\usepackage{amsfonts}
\usepackage[margin=1in]{geometry}
\usepackage{graphicx}
\usepackage{multirow}
\usepackage{multicol}
\usepackage{booktabs}
\usepackage{xcolor}
\usepackage{indentfirst}
\usepackage{amsthm}
\usepackage{setspace}
\usepackage{bm}
\usepackage[notquote]{hanging}
\usepackage{amsthm}
\usepackage{rotating}
\usepackage{url}
\newcommand\independent{\protect\mathpalette{\protect\independenT}{\perp}}
\def\independenT#1#2{\mathrel{\rlap{$#1#2$}\mkern2mu{#1#2}}}

\newtheorem{theorem}{Theorem}
\newtheorem{lemma}[theorem]{Lemma}

\allowdisplaybreaks

\begin{document}

\title{Inferring the minimum spanning tree from a sample network}
\author{Jonathan Larson and Jukka-Pekka Onnela}
\date{\today}

\maketitle

\doublespacing

\begin{abstract}
Minimum spanning trees (MSTs) are used in a variety of fields,
from computer science to geography.
Infectious disease researchers have used them
to infer the transmission pathway of certain pathogens.
However, these are often the MSTs of sample networks,
not population networks,
and surprisingly little is known about what can be inferred about a population
MST from a sample MST.
We prove that if $n$ nodes (the sample) are selected uniformly at random from
a complete graph with $N$ nodes and
unique edge weights (the population),
the probability that an edge is in the population graph's MST
given that it is in the sample graph's MST is $\frac{n}{N}$.
We use simulation to investigate this conditional probability
for $G(N,p)$ graphs, Barab\'{a}si–Albert (BA) graphs,
graphs whose nodes are distributed in $\mathbb{R}^2$
according to a bivariate standard normal distribution,
and an empirical HIV genetic distance network.
Broadly, results for the complete, $G(N,p)$, and normal graphs
are similar,
and results for the BA and empirical HIV graphs are similar.
We recommend that researchers use an edge-weighted random walk
to sample nodes from the population
so that they maximize the probability that an edge is
in the population MST given that it is in the sample MST.

\end{abstract}

\section{Introduction}
\label{sec:intro}

A graph consists of nodes (also called vertices) and edges,
with each edge connecting a pair of nodes.
A tree is a subset of the edges of a connected graph
that has no cycles,
i.e., there is only one path from one node to any other node.
A spanning tree connects all the vertices of the graph,
and the minimum spanning tree (MST)
is the spanning tree with the lowest total edge weight. 
If the original graph is not connected,
its minimum spanning forest (MSF) consists of the MSTs
of its connected components.
If the edge weights are unique, there is only one MST;
if there are duplicate edge weights,
there may be more than one MST.
Given a weighted graph,
there are a variety of algorithms for obtaining its MST.
The classic algorithms are
Bor\r{u}vka's \cite{nesetril2001},
Prim's \cite{prim1957},
and Kruskal's \cite{kruskal1956}.

Researchers in a variety of fields have used MSTs
to analyze network data.
For example,
neurologists have used MSTs to compare brain networks,
in which regions of the brain are nodes and edges denote
structural or functional connections
\cite{tewarie2015,vandellen2018}.
Computer scientists have used MSTs to segment video
into meaningful partitions \cite{wang2020}
and decompose images into a base layer and a detail layer
\cite{jin2020}. 
Geographers have used MSTs to describe local building patterns
\cite{wu2018}.
% MSTs are used in multiple clustering algorithms
% \cite{mishra2020}.
The use of MSTs to study the hierarchical structure of financial
markets using correlation-based networks was first proposed
in \cite{mantegna1999}; 
some of these concepts were later expanded in a series of papers
\cite{onnela2002,onnela2003} that included the application of
MST to a subset of 116 stocks of the 500 stocks in the S\&P 500 index.
In 2020,
PNAS published at least six articles in which researchers
used an MST
\cite{li2020,steinbrenner2020,manning2020,saul2020,matsumura2020,hahn2020}.

The MSTs we construct are usually only the MSTs of sample networks,
whereas our interest lies in characterizing the MST of the population.
Despite the importance of the problem,
we are not aware of any published work on
what may be inferred about the MST of the population
network from the MST of a sample network.
(We should clarify that when we say sampling,
we mean the sampling of nodes.)
Instead, research has focused on other issues,
such as using knowledge of the population network
to predict the behavior of trees that span sample subgraphs
\cite{bertsimas1990};
using knowledge of the population network and sampling process
to find a set of edges that contains the sample MST with high
probability \cite{goemans2006}; or
finding the MST when edge weights are random
\cite{torkestani2012}.

This paper aims to answer the following related questions:
\begin{enumerate}
\item Given that an edge is in the sample graph
but not the sample MST,
what is the probability that it is not in the population MST?
We can think of this probability as the negative predictive value (NPV).
\item Given that an edge appears in the sample MST,
what is the probability that it appears in the population MST?
We can think of this probability as the positive predictive value (PPV).
\item How well can we estimate these probabilities by bootstrapping
from the sample graph?
\item How strong is the relationship between the number of bootstrap MSTs
that an edge is in and whether or not that edge is in the population MST?
\end{enumerate}

\section{Theory}
\label{sec:theory}

A graph $G = (V,E)$ consists of a set of nodes (or vertices) $V$
and a set of edges $E$;
each edge $e\in E$ connects a pair of nodes $u,v\in V$,
$u \ne v$,
so that we may write $e = (u,v)$.
Here, we assume all graphs are undirected,
so $(u,v) = (v,u)$.
Each edge $e\in E$ has a weight $w(e) > 0$.
A cycle $C$ is a set of edges $\{e_1,\dotsc,e_k\}\subset E$
such that $\forall\, i\in\{1,\dotsc,k-1\}$,
$e_i = (v_i,v_{i+1})$,
where $i \ne j \implies v_i \ne v_j$,
and $e_k = (v_k,v_1)$.
In some texts a cycle also contains the associated vertices,
but here, for ease of exposition,
the terms ``cycle'', ``tree'', ``spanning tree'',
``MST'', and ``cut'' all refer to sets of edges.
A tree $T$ is a subset of the edges of a connected graph
that contains no cycles;
$T$ is spanning if $\forall\, v\in V$, $\exists\, e\in T$
such that $v$ is an endpoint of $e$;
and a minimum spanning tree (MST) has the lowest total edge weight
of all spanning trees.
If the original graph is not connected,
its minimum spanning forest (MSF) consists of the MSTs
of its connected components.
If $V_1$ and $V_2$ are non-empty sets of nodes satisfying $V_1\cup V_2 = V$
and $V_1\cap V_2 = \emptyset$,
then the associated cut $D$ consists of edges connecting
one node from $V_1$ and one node from $V_2$.
In symbols,
$D = \{(v_1,v_2)\in E : v_1\in V_2,\ v_2\in V_2\}$.

The symbol $|E|$ denotes the number of elements in the set $E$
and $A\triangle B = (A\setminus B)\cup(B\setminus A)$.
Theorems \ref{theorem:st} and \ref{theorem:unique}
are known facts but are proved here for completeness.
Theorem \ref{theorem:excluded} is a version of the cycle property
and Theorem \ref{theorem:included} is a version of the cut property.
Theorem \ref{theorem:bijective} is based on \cite{raphael2019}.

\begin{lemma}
\label{lemma:degree}
If $G$ is a connected graph with $N \in \{2,3,4,\dotsc\}$ nodes
and $T$ is a spanning tree of $G$,
then at least one node in $G$ is the endpoint of only one edge in $T$.
\end{lemma}
\begin{proof}
Suppose the contrary,
that each node in $G$ is the endpoint of at least two edges in $T$.
(No node in $G$ can be the endpoint of zero edges in $T$
because $T$ is spanning, and thus connects all nodes.)
Start at any node in $G$ and walk along an edge in $T$.
From the next node, walk along a different edge in $T$.
Continue this walk,
leaving each node by a different edge than the one by which you arrived.
Since $N < \infty$,
at some point you will arrive at a node you have already visited.
This means that $T$ contains a cycle,
which is a contradiction.
So at least one node in $G$ is the endpoint of only one edge in $T$.
\end{proof}

\begin{theorem}
\label{theorem:st}
If $G$ is a connected graph with $N \in \{2,3,4,\dotsc\}$ nodes
and $T$ is a spanning tree of $G$,
then $|T| = N - 1$.
\end{theorem}
\begin{proof}
Suppose $N = 2$.
Then $G$ has only one edge,
$T = E$, and $|T| = 1 = N - 1$.
Now suppose the theorem holds for graphs with $N - 1 \ge 2$ nodes
and suppose $G$ has $N$ nodes.
Find a node $v$ in $G$ that is the endpoint of only one edge $e$ in $T$.
Remove $v$ from $G$ to create the new graph $G'$
and remove $e$ from $T$ to create the new tree $T'$.
Since $T'$ is a spanning tree of $G'$,
and $G'$ has $N - 1$ nodes,
$|T'| = N - 2$.
Since $|T'| = |T| - 1$, $|T| = N - 1$.
Thus, through induction,
we have shown that the theorem is true for $N\in\{2,3,4,\dotsc\}$.
\end{proof}

We can extend Theorem \ref{theorem:st} to conclude that,
if $G$ is a graph with $N \in \{2,3,4,\dotsc\}$ nodes
and $K$ components,
and $T$ is a spanning tree of $G$,
then $|T| = N - K$.

\begin{theorem}
\label{theorem:unique}
If $G = (V,E)$ is a connected graph with $N < \infty$ nodes
and unique positive edge weights
then $G$ has exactly one MST.
\end{theorem}
\begin{proof}
If $N \le 1$ then $E = \emptyset$ and the MST is empty.
If $N = 2$ then $|E| = 1$ and the MST is $E$.
Suppose $N \ge 3$.
Suppose the opposite of the statement of the theorem,
that $G$ has more than one MST.
Let $A$ and $B$ denote two distinct MSTs of $G$.
Let $a = \arg\min_{e\in A\triangle B} w(e)$.
Since the edge weights are unique, so is $a$.
Without loss of generality, assume $a\in A$.
Since $B$ is a spanning tree,
$B\cup\{a\}$ contains a cycle $C$ containing $a$.
Since $A$ is a tree,
$A$ cannot contain $C$,
meaning $C$ must contain an edge $b \notin A$.
Since $a,b\in A\triangle B$ and $a = \arg\min_{e\in A\triangle B} w(e)$,
$w(b) > w(a)$.
This means $B\cup\{a\}\setminus\{b\}$
is a spanning tree with lower total edge weight than $B$,
which is a contradiction.
Thus $G$ has exactly one MST.
\end{proof}

\begin{theorem}
\label{theorem:excluded}
Suppose $G = (V,E)$ is a connected graph with $N < \infty$ nodes
and edges with unique positive weights.
Let $T$ be the (unique) MST of $G$.
Then $e\in E\setminus T$ if and only if
$e$ belongs to a cycle $C$ in $G$ and $e$ has greater weight
than every other edge in $C$.
\end{theorem}
\begin{proof}
Suppose $e\in E\setminus T$.
Then $T\cup\{e\}$ contains a cycle $C$.
The weight of $e$ cannot be equal to the weight of any other edge in $C$
because all the edge weights are unique.
If $\exists\, e'\in C$ such that $w(e') > w(e)$,
then $T\cup\{e\}\setminus\{e'\}$ would be a spanning tree
with smaller total edge weight than $T$,
and $T$ would not be an MST.
Thus $e$ has edge weight greater than every other edge in $C$.

Suppose $e$ belongs to a cycle $C$ in $G$ and $e$ has greater weight
than every other edge in $C$.
Suppose $e\in T$.
If $e'$ is any other edge in $C$ then $T\cup\{e'\}\setminus\{e\}$
would be a spanning tree with smaller total edge weight than $T$,
which is a contradiction.
Thus $e\notin T$.
\end{proof}

\begin{theorem}
\label{theorem:included}
Suppose $G = (V,E)$ is a connected graph with $N < \infty$ nodes
and edges with unique positive weights.
Let $T$ be the (unique) MST of $G$.
Then $e\in T$ if and only if
$e$ belongs to a cut $D$ in $G$ and $e$ has lower weight
than every other edge in $D$.
\end{theorem}
\begin{proof}
Suppose $e\in T$
but each cut containing $e$ contains another edge with lower weight
than $e$.
Removing $e$ from $T$ would split $T$ into two components,
$T_1 $ and $T_2$,
where $T = T_1 \cup T_2 \cup \{e\}$.
Let $V_1$ denote the set of endpoints of edges in $T_1$,
let $V_2$ denote the set of endpoints of edges in $T_2$,
and let $D$ denote the set of edges in $E$
with one endpoint in $V_1$ and the other endpoint in $V_2$.
Let $e'$ denote another edge in $D$ with lower weight than $e$.
Then $T\cup\{e'\}\setminus\{e\}$ is a spanning tree
with lower total edge weight than $T$,
a contradiction.
Thus if $e\in T$ then $e$ belongs to a cut $D$
and has the lowest weight of any edge in $D$.

Suppose $e = (v_1,v_2)$ belongs to a cut $D$
in $G$ and $e$ has lower weight
than every other edge in $D$.
Let $V_1$ and $V_2$ denote the two sets of vertices separated
by this cut,
with $v_1 \in V_1$ and $v_2\in V_2$.
If $e\notin T$
then $T\cup\{e\}$ contains a cycle $C$.
$C\setminus\{e\}$ is a path from $v_1 \in V_1$ to $v_2 \in V_2$,
and thus contains an edge $e' \in D$.
By assumption, $w(e) < w(e')$.
Thus $T\cup\{e\}\setminus\{e'\}$
is a spanning tree with lower total edge weight than $T$,
a contradiction.
Thus $e\in T$.
\end{proof}

At this point it is necessary to define more rigorously our first
quantity of interest, the negative predictive value (NPV).
We consider two related but not necessarily equivalent approaches.
First, let $G_n = (V_n,E_n)$ be the subgraph of $G = (V,E)$ induced
by sampling $n$ nodes from $V$.
(At this point we do not specify the sampling mechanism.)
Assuming $G$ has unique edge weights,
let $T$ be the unique MST of $G$
and let $T_n$ be the unique MST of $G_n$.
Finally, let $e$ be an edge selected uniformly at random from $E$.
We want to know $P(e\in E\setminus T|e\in E_n\setminus T_n)$.
Of course, this quantity is only defined if
$P(e\in E_n\setminus T_n) > 0$.
The second approach is to find
\[
E\left(\frac{\left|E_n\setminus (T\cup T_n)\right|}{\left|E_n\setminus T_n\right|} I\left(\left|E_n\setminus T_n\right| > 0\right)\right) \text{,}
\]
where $I(A) = 1$ if event $A$ transpires and $0$ otherwise.

\begin{theorem}
\label{theorem:npv}
Let $G = (V,E)$ be a graph with $N < \infty$ nodes,
unique positive edge weights,
and MSF $T$.
Let $G_n = (V_n,E_n)$ be a subgraph of $G$ with $n\in\{0,1,\dotsc,N\}$ nodes
and MSF $T_n$.
Then $T \cap E_n\setminus T_n = \emptyset$.
\end{theorem}
\begin{proof}
Suppose $\exists\, e\in T \cap E_n\setminus T_n$.
Then $\{e\}\cup T_n$ contains a cycle $C$.
If there exists an edge $e'\in C$ with greater weight than $e$
then $T_n\cup\{e\}\setminus\{e'\}$
is a spanning tree with smaller weight than $T_n$,
a contradiction.
Thus $e$ has the largest weight of any edge in $C$.
Since $C\subset E_{n}\subset E$,
Theorem \ref{theorem:excluded} implies $e\notin T$.
This is a contradiction,
so $T \cap E_n\setminus T_n = \emptyset$.
\end{proof}
Theorem \ref{theorem:npv} implies that
$E_n\setminus T_n = E_n\setminus (T\cup T_n)$, meaning
\[
P(e\in E\setminus T|e\in E_n\setminus T_n) = \frac{P(e\in E_n\setminus (T\cup T_n))}{P(e\in E_n\setminus T_n)} = 1
\]
and 
\[
E\left(\frac{\left|E_n\setminus (T\cup T_n)\right|}{\left|E_n\setminus T_n\right|} I\left(\left|E_n\setminus T_n\right| > 0\right)\right) = 1 \text{.}
\]
In other words, for both approaches, the NPV is 1.
This is irrespective of how $G$ is generated
or how the nodes in $V_n$ are sampled;
we just require that the edge weights be unique
and that the appropriate denominators are non-zero.

Our next task is to find the positive predictive value, or PPV.
Just as with the NPV, we take two approaches.
We want to find
\[
P\left(e\in T\middle| e\in T_n\right) \quad \text{and} \quad E\left(\frac{|T\cap T_n|}{|T_n|} I(|T_n| > 0)\right) \text{.}
\]
Unlike with the NPV,
these values will depend on how $G$ is generated
and how the nodes in $V_n$ are sampled.
We begin with the case where $G$ is a complete graph
and the nodes in $V_n$ are sampled uniformly at random from $V$.

\begin{theorem}
\label{theorem:complete}
Let $G = \left(V,E\right)$ be a complete graph
with $N < \infty$ nodes and positive, unique edge weights.
Let $G_n = \left(V_n,E_n\right)$,
where $V_n$ contains $n\in\{2,3,\dotsc,N\}$
nodes selected uniformly at random from $V$,
and where $E_n$ contains those edges from $E$
that have both endpoints in $V_n$.
(In other words,
$G_n$ is the subgraph of $G$ induced by the nodes in $V_n$.)
Define $T$ to be the unique MST of $G$
and define $T_n$ to be the unique MST of $G_n$.
Then
$P\left(e\in T \middle|e\in T_n \right) = E\left(\frac{|T\cap T_n|}{|T_n|} I(|T_n| > 0)\right) = \frac{n}{N}$.
\end{theorem}
\begin{proof}
Since $G$ is complete, $G_n$ must be connected,
so $|T_n| = n - 1$ and
\[
P(e\in T_n) = \frac{|T_n|}{|E|} = \frac{n - 1}{\binom{N}{2}} \text{.}
\]
Since the nodes in $V_n$ are selected uniformly at random from $V$,
without respect to whether they are the endpoints of edges in $T$,
$T \independent E_n$.
Thus
\begin{align*}
P\left(e\in T\cap T_n\right) &= P\left( e\in T\cap T_n\cap E_n\right) & (1) \\
&= P\left( e\in T\cap E_n\right) & (2) \\
&= P(e\in T) P(e\in E_n) \\
&= \frac{|T|}{|E|} \frac{|E_n|}{|E|} \\
&= \frac{N-1}{\binom{N}{2}} \frac{\binom{n}{2}}{\binom{N}{2}} \text{.} \\
\end{align*}
Note that we used Theorem \ref{theorem:npv} to move from (1) to (2).
So
\begin{align*}
P\left(e\in T \middle|e\in  T_n\right) &= \frac{P\left(e\in T \cap T_n\right)}{P\left(e\in T_n\right)} \\
% &= \frac{N-1}{\binom{N}{2}} \frac{\binom{n}{2}}{\binom{N}{2}} \frac{\binom{N}{2}}{n-1} \\
% &= \frac{N-1}{n-1} \frac{\binom{n}{2}}{\binom{N}{2}} \\
% &= \frac{N-1}{n-1} \frac{n(n-1)}{N(N-1)} \\
&= \frac{n}{N} \text{.} \\
\end{align*}

Since $G$ is connected and $n\ge 2$, $|T_n| = n - 1 > 0$,
so
\[
E\left(\frac{|T\cap T_n|}{|T_n|} I(|T_n| > 0)\right) = E\left(\frac{|T\cap T_n|}{|T_n|} \right) = \frac{E\left(|T\cap T_n|\right)}{n - 1} \text{.} 
\]
Using Theorem \ref{theorem:npv} again,
$E\left(|T\cap T_n|\right) = E\left(|T\cap E_n|\right)$.
If $e_1,\dotsc,e_{N-1}$ is an enumeration of the edges in $T$,
then
\[
E\left(|T\cap E_n|\right) = \sum_{i=1}^{N-1} E\left[I(e_i \in E_n)\right] = \sum_{i=1}^{N-1} \frac{\binom{n}{2}}{\binom{N}{2}} = (N-1) \frac{\binom{n}{2}}{\binom{N}{2}}
\]
and
\[
E\left(\frac{|T\cap T_n|}{|T_n|} I(|T_n| > 0)\right) = \frac{N-1}{n-1} \frac{\binom{n}{2}}{\binom{N}{2}} = \frac{n}{N} \text{.}
\]
\end{proof}
In other words,
under the conditions of Theorem \ref{theorem:complete},
the probability that an edge is in the population MST
given that it is in the sample MST is equal to the proportion
of the population that has been sampled.
Applied researchers can increase this probability by
recruiting more participants,
and the increase is linear in sample size.

Theorem \ref{theorem:complete} relies on two key facts:
The first is that $|T|$, $|T_n|$, and $|E_n|$ are known constants,
which results from $G$ being complete.
The second is that $T$ is independent of $E_n$,
which results from sampling the nodes uniformly at random.
In Theorem \ref{theorem:gnp},
we consider a scenario where $|T|$, $|T_n|$, and $|E_n|$ are random,
but $T$ is still independent of $E_n$.

\begin{theorem}
\label{theorem:gnp}
Let $G$ and $G_n$ be defined as in Theorem \ref{theorem:complete}.
Let $G' = \left(V,E'\right)$,
where each edge from $E$ is included in $E'$
independently and with probability $p$.
Let $G_n' = \left(V_n,E_n'\right)$,
where $E_n'$ contains those edges from $E'$
that have both endpoints in $V_n$.
(In other words,
$E_n' = E_n \cap E'$.)
Define $T'$ to be the unique MSF of $G'$
and $T_n'$ to be the unique MSF of $G_n'$.
Let $K'$ denote the number of components in $G'$
and $K_n'$ denote the number of components in $G_n'$.
Then
\[
P\left( e\in T' \middle|e\in  T_n' \right) = \frac{n}{N} \left(\frac{n-1}{N-1}\right) \left(\frac{N - E\left(K'\right)}{n - E\left(K_n'\right)}\right) \text{.}
\]
\end{theorem}
\begin{proof}
Note that $T_n' \subset E_n'$
and $\left( e \in T' \independent  e \in E_n'\right)|  e \in E'$.
That is, if an edge is in $E'$,
whether or not it is in $T'$ has no bearing on whether or not
it will be included in $E_n'$.
Using this fact and Theorem \ref{theorem:npv},
\begin{align*}
P\left(e\in T'\cap T_n'\right) &= P\left(e\in  T'\cap E_n'\right) \\
&= P\left(e\in  T'\cap E_n'\cap E'\right) \\
&= P\left(e\in  T'\cap E_n'\middle|e\in   E'\right) P\left(e\in   E'\right) \\
&= P\left(e\in  T'\middle|e\in   E'\right) P\left(e\in   E_n'\middle|  e\in E'\right) P\left( e\in  E'\right) \\
&= P\left(e\in  T'\cap E'\right) P\left( e\in  E_n'\middle| e\in  E'\right) \\
&= P\left(e\in  T'\right) P\left( e\in  E_n'\middle|e\in   E'\right) \\
&= P\left(e\in  T'\right) P\left( e\in  E_n\cap E'\middle|e\in   E'\right) \\
&= P\left(e\in  T'\right) \frac{P\left( e\in  E_n\cap E'\right)}{P\left( e\in E'\right)} \\
&= P\left(e\in  T'\right) \frac{P\left( e\in  E_n\right) P\left(   e\in E'\right)}{P\left(e\in  E'\right)} \\
&= P\left(e\in  T'\right) P\left(e\in   E_n\right) \text{.} \\
\end{align*}
Next,
\begin{align*}
P\left(e\in  T'\right) %&= \sum_{k=1}^{N}  P\left(e\in  T', K' = k\right) \\
&= \sum_{k=1}^{N}  P\left(e\in  T'\middle| K' = k\right) P\left( K' = k\right) \\
&= \sum_{k=1}^{N}  \frac{N-k}{\binom{N}{2}} P\left( K' = k\right) \\
&= \frac{1}{\binom{N}{2}} \left[N - E\left(K'\right)\right] \text{.} \\
\end{align*}
Similarly,
\[
P\left(e\in  T_n'\right) = \frac{1}{\binom{n}{2}} \left[n - E\left(K_n'\right)\right] \text{.}
\]
Thus,
\begin{align*}
P\left(e\in  T' \middle| e\in  T_n' \right) &= \frac{P\left(e\in  T'\cap T_n'\right)}{P\left(e\in  T_n'\right)} \\
&= \frac{P\left(e\in  T'\right) P\left( e\in  E_n\right)}{P\left(e\in  T_n'\right)} \\
&= \frac{N - E\left(K'\right)}{n - E\left(K_n'\right)} \frac{n(n-1)}{N(N-1)} \text{.} \\
\end{align*}
\end{proof}
If $E\left(K'\right)$ and $E\left(K_n'\right)$ approach $1$
as $n$ and $N$ increase toward infinity
then $P\left(e\in T' \middle|e\in  T_n' \right)\to\frac{n}{N}$,
which is the result for the complete graph.
We were unable to determine
$E\left(\frac{|T\cap T_n|}{|T_n|} I(|T_n| > 0)\right)$
analytically,
so we explored it through simulation
(described in the next section).

At this point we turn to graphs with more than one MST.
%Theorem \ref{theorem:bijective} is based on \cite{raphael2019}.

\begin{theorem}
\label{theorem:bijective}
Suppose $G = (V,E)$ is a connected, weighted graph
with finitely many nodes
and more than one MST.
Let $A$ and $B$ denote two of these MSTs.
Then there exists a bijective
function $g : A\setminus B \to B \setminus A$
such that the weight of $a$ is equal to the weight of $g(a)$,
$a$ and $g(a)$ belong to a cycle $C \subset B\cup\{a\}$
in which they are the maximum weight edges,
and
$a$ and $g(a)$ belong to a cut $D\subset E$
in which they are the minimum weight edges.
\end{theorem}
\begin{proof}
Take $a_1 = (v_1,v_2)\in A\setminus B$.
(If $A\setminus B = \emptyset$ then $A\subset B$,
and since $|A| = |B|$,
that would imply that $A = B$.)
Removing $a_1$ from $A$ would split $A$ into two components,
$A_1 $ and $A_2$,
where $A = A_1 \cup A_2 \cup \{a_1\}$.
Let $V_1$ denote the set of endpoints of edges in $A_1$,
let $V_2$ denote the set of endpoints of edges in $A_2$,
with $v_1\in V_1$ and $v_2\in V_2$,
and let $D_1$ denote the set of edges in $E$
with one endpoint in $V_1$ and the other endpoint in $V_2$.
Since $a_1\notin B$,
$B\cup\{a_1\}$ contains a cycle $C_1$.
Since $C_1\setminus\{a_1\}$ consists of a path from $v_1\in V_1$
to $v_2\in V_2$,
there is at least one edge $b_1 \in C_1\setminus\{a_1\}$
that is also in $D_1$.
Since $D_1 \cap A = \{a_1\}$ and $b_1 \ne a_1$,
$b_1 \notin A$.
So $b_1\in B\setminus A$.
\begin{enumerate}
\item Suppose the weight of $b_1$ is strictly greater than the weight of $a_1$.
Then $B\cup \{a_1\}\setminus \{b_1\}$
would be a spanning tree with total edge weight less
than $B$,
and $B$ would not be an MST.
\item Suppose the weight of $b_1$ is strictly less than the weight of $a_1$.
Then $A\cup \{b_1\}\setminus \{a_1\}$
would be a spanning tree with total edge weight less than $A$,
and $A$ would not be an MST.
\end{enumerate}
So $w(a_1) = w(b_1)$.
\begin{enumerate}
\item Suppose $C_1$ contains an edge $c$ with weight greater than that
of $a_1$ and $b_1$.
Then $B\cup\{a_1\}\setminus\{c\}$
is a spanning tree with lower total edge weight than $B$,
which is a contradiction.
Thus $a_1$ and $b_1$ have the maximum weight of any edge in $C_1$.
\item Suppose $D_1$ contains  an edge $d$ with weight less than that
of $a_1$ and $b_1$.
Then $A\cup\{d\}\setminus\{a_1\}$
is a spanning tree with lower total edge weight than $A$,
which is a contradiction.
Thus $a_1$ and $b_1$ have the minimum weight of any edge in $D_1$.
\end{enumerate}

Define $T_1 = A\cup\{b_1\}\setminus\{a_1\}$.
It is a spanning tree with total edge weight equal to $A$,
so it is an MST.
If $T_1 \ne B$,
repeat the process:
select $a_2\in T_1\setminus B$
and $b_2\in B\setminus T_1$
such that $a_2$ and $b_2$ belong to
\begin{enumerate}
\item a cycle $C_2 \subset B \cup \{a_2\}$
in which they have the maximum weight of any edge; and
\item a cut $D_2 \subset E$
in which they have the minimum weight of any edge.
\end{enumerate}
Note that the edges $a_1$, $a_2$, $b_1$, and $b_2$
are all distinct.
Define $T_2 = T_1\cup\{b_2\}\setminus\{a_2\}$,
which is another MST.
If $T_2 \ne B$,
repeat the process.
If $k = |A\setminus B| = |B\setminus A|$,
then $T_k = B$.
\end{proof}

Theorem \ref{theorem:bijective} implies that any graph with multiple MSTs
must contain a cycle with two edges sharing the maximum weight.
The converse is not true,
as demonstrated in Figure \ref{figure:counter}(a).
Even if a graph contains a cycle with two edges sharing the maximum weight,
it may only have one MST.

\begin{figure}
\centering
\includegraphics[width = 0.8\linewidth]{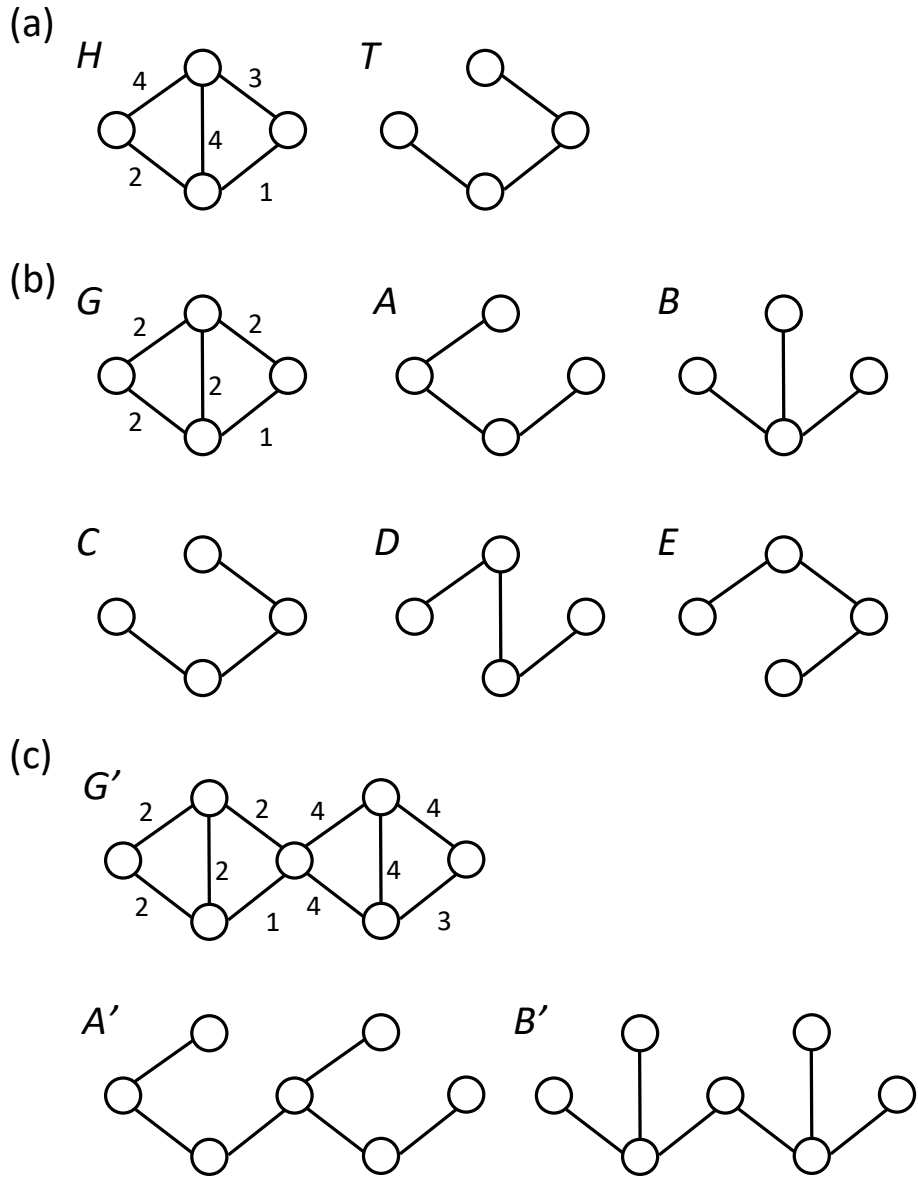}
\caption{Counterexamples.
(a) $H$ contains a cycle with two edges sharing the maximum weight,
but it only has one MST $T$.
(b) For the graph $G$,
there are $4! = 24$ different orderings of the edges,
but $5$ MSTs (labeled $A$ through $E$).
Since $24 / 5$ is not an integer,
there cannot be an equal number of orderings per MST.
In fact,
$A$ has four orderings leading to it
but each of the other MSTs has five orderings leading to it.
(c) If we chain $G$ to create $G'$,
we see that the MST $A'$ has $4\times 4 = 16$ orderings leading to it
and the MST $B'$ has $5\times 5 = 25$ orderings leading to it.
Thus the ratio of the number of orderings leading to $A'$
to the number of orderings leading to $B'$ is $16/25 = (4/5)^2$.
Chaining $G$ in this way indefinitely demonstrates that even asymptotically
the ratio of orderings leading to each MST does not approach $1$.}
\label{figure:counter}
\end{figure}

Let $G = (V,E)$ be a connected, weighted graph with $N$ vertices,
$m$ edges,
and at least two MSTs,
$A$ and $B$.
Then $m \ge N$ and
there must be at least two edges in $E$ with the same weight.
Label the edges in $E$ from lowest to highest weight
so that, if $w(e_i)$ is the weight of the $i$th edge,
$w(e_1) \le w(e_2) \le \cdots \le w(e_m)$.
For edges with the same weight,
order them arbitrarily.
Suppose there are $K \ge 1$ weights that are shared by more than one edge,
with $k_i \ge 2$ edges sharing the $i$th weight
for $i = 1,\dotsc,K$.
%with $k_1 \ge 2$ edges sharing the first weight,
%$k_2 \ge 2$ edges sharing the second weight,
%. . . and $k_K \ge 2$ edges sharing the $K$th weight.
Then there are $\prod_{i=1}^K k_i!$ different orderings
$\sigma$ of the edges such that
$w(e_{\sigma(1)}) \le w(e_{\sigma(2)}) \le \cdots \le w(e_{\sigma(m)})$.
Note that the identity function $\sigma(i) = i$
is counted as one of these orderings.

In order to find the MST of a graph with unique
edge weights,
the actual values of the weights are not important.
It is only their order that matters.
If we think of each ordering $\sigma$
as treating edge $e_{\sigma(i)}$ as having weight
$w_\sigma\left(e_{\sigma(i)}\right) = i$,
then the edge weights are unique
and $\sigma$ results in a unique MST.
If the number of orderings leading to each MST were equal,
we could sample uniformly
from the set of MSTs of a graph
by sampling uniformly from the set of edge orderings.
Unfortunately, Figure \ref{figure:counter}(b)
demonstrates that the number of orderings leading to each MST
may not be equal.
For the graph $G$,
there are $4! = 24$ different orderings of the edges,
but $5$ MSTs (labeled $A$ through $E$).
Since $24 / 5$ is not an integer,
there cannot be an equal number of orderings per MST.
In fact,
$A$ has four orderings leading to it
but each of the other MSTs has five orderings leading to it.
It is tempting to assume that as the number of edges and/or cycles
approaches infinity,
the proportion of orderings leading to each MST
will approach the same value,
but Figure \ref{figure:counter}(c) provides a counterexample.
If we chain $G$ to create $G'$,
we see that the MST $A'$ has $4\times 4 = 16$ orderings leading to it
and the MST $B'$ has $5\times 5 = 25$ orderings leading to it.
Thus the ratio of the number of orderings leading to $A'$
to the number of orderings leading to $B'$ is $16/25 = (4/5)^2$.
As the number of instances of $G$ chained together approaches infinity,
the ratio of the number of orderings leading to the $A$ chain
to the number of orderings leading to the $B$ chain approaches 0,
not 1.
So the proportion of orderings leading to each MST
does not approach the same value
even as the number of edges and/or cycles
approaches infinity.

\section{Methods}

Other types of graphs and sampling methods are not as tractable
as complete graphs and sampling uniformly at random.
In order to estimate the probability that an edge is in the population MST
given that it is in the sample MST for more complex situations,
a simulation study was conducted.
For the simulation study,
we targeted the parameter
$E\left(\frac{|T\cap T_n|}{|T_n|} I(|T_n| > 0)\right)$.

\subsection{Simulation Study}
\label{sec:sim}

The following algorithm was repeated for $i = 1,\dotsc,1000$.
Figure \ref{figure:schematic} contains a schematic of one replication.
\begin{enumerate}
\item Generate a weighted graph $g_i$ with $N = 100$ nodes.
(More information on the type of graph is included below.)
This graph $g_i$ is the population graph.
\item Find $t_{g_i}$, the MST of $g_i$.
Thus $t_{g_i}$ is the population MST.
\item Sample $n$ nodes from $g_i$,
yielding the induced subgraph $h_i$.
(More information on the sampling process and the value of $n$
is included below.)
Thus $h_i$ is the sample graph.
\item Find $t_{h_i}$, the MST of $h_i$.
Thus $t_{h_i}$ is the sample MST.
\item Calculate the positive predictive value
\[
\text{PPV}_i = \frac{\text{\# of edges in $t_{g_i}\cap t_{h_i}$}}{\text{\# of edges in $t_{h_i}$}} \text{.}
\]
\item Repeat the following for $j = 1,\dotsc,100$:
\begin{enumerate}
\item Sample $n^2/N$ nodes from $h_i$,
yielding the induced subgraph $h_{i,j}$.
In other words, sample the same proportion of nodes from $h_i$
as were sampled from $g_i$.
Thus $h_{i,j}$ is the bootstrap sample graph.
\item Find $t_{h_{i,j}}$, the MST of $h_{i,j}$.
Thus $t_{h_{i,j}}$ is the bootstrap MST.
\item Calculate the bootstrap positive predictive value
\[
\text{BPPV}_{i,j} = \frac{\text{\# of edges in $t_{h_i}\cap t_{h_{i,j}}$}}{\text{\# of edges in $t_{h_{i,j}}$}} \text{.}
\]
\end{enumerate}
\item Calculate
\[
\overline{\text{BPPV}}_i = \frac{1}{100} \sum_{j=1}^{100} \text{BPPV}_{i,j} \text{.}
\]
\item Calculate the area under the ROC curve (AUC$_i$)
using the number of times that an edge appears in a bootstrap MST
(i.e., one of the $t_{h_{i,j}}$) as the predictor
and whether that edge appears in $t_{g_i}$ as the outcome.
\end{enumerate}
The following statistics were calculated to summarize the 1,000 replications:
\begin{align*}
\overline{\text{PPV}} &= \frac{1}{1000} \sum_{i=1}^{1000} \text{PPV}_i \\
\overline{\overline{\text{BPPV}}} &= \frac{1}{1000} \sum_{i=1}^{1000} \overline{\text{BPPV}}_i \\
\overline{\text{AUC}} &= \frac{1}{1000} \sum_{i=1}^{1000} \text{AUC}_i \text{.} \\
\end{align*}
Confidence intervals were calculated as follows:
\begin{align*}
\overline{\text{PPV}} &\pm z_{0.975} \sqrt{\frac{1}{1000} \left(\frac{1}{999}\sum_{i=1}^{1000} \left(\text{PPV}_i - \overline{\text{PPV}}\right)^2\right)} \\
\overline{\overline{\text{BPPV}}} &\pm z_{0.975} \sqrt{\frac{1}{1000} \left(\frac{1}{999}\sum_{i=1}^{1000} \left(\overline{\text{BPPV}}_i - \overline{\overline{\text{BPPV}}}\right)^2\right)} \\
\overline{\text{AUC}} &\pm z_{0.975} \sqrt{\frac{1}{1000} \left(\frac{1}{999} \sum_{i=1}^{1000} \left(\text{AUC}_i - \overline{\text{AUC}}\right)^2\right)} \text{.} \\
\end{align*}

\begin{figure}
\centering
\includegraphics[width = 0.8\linewidth]{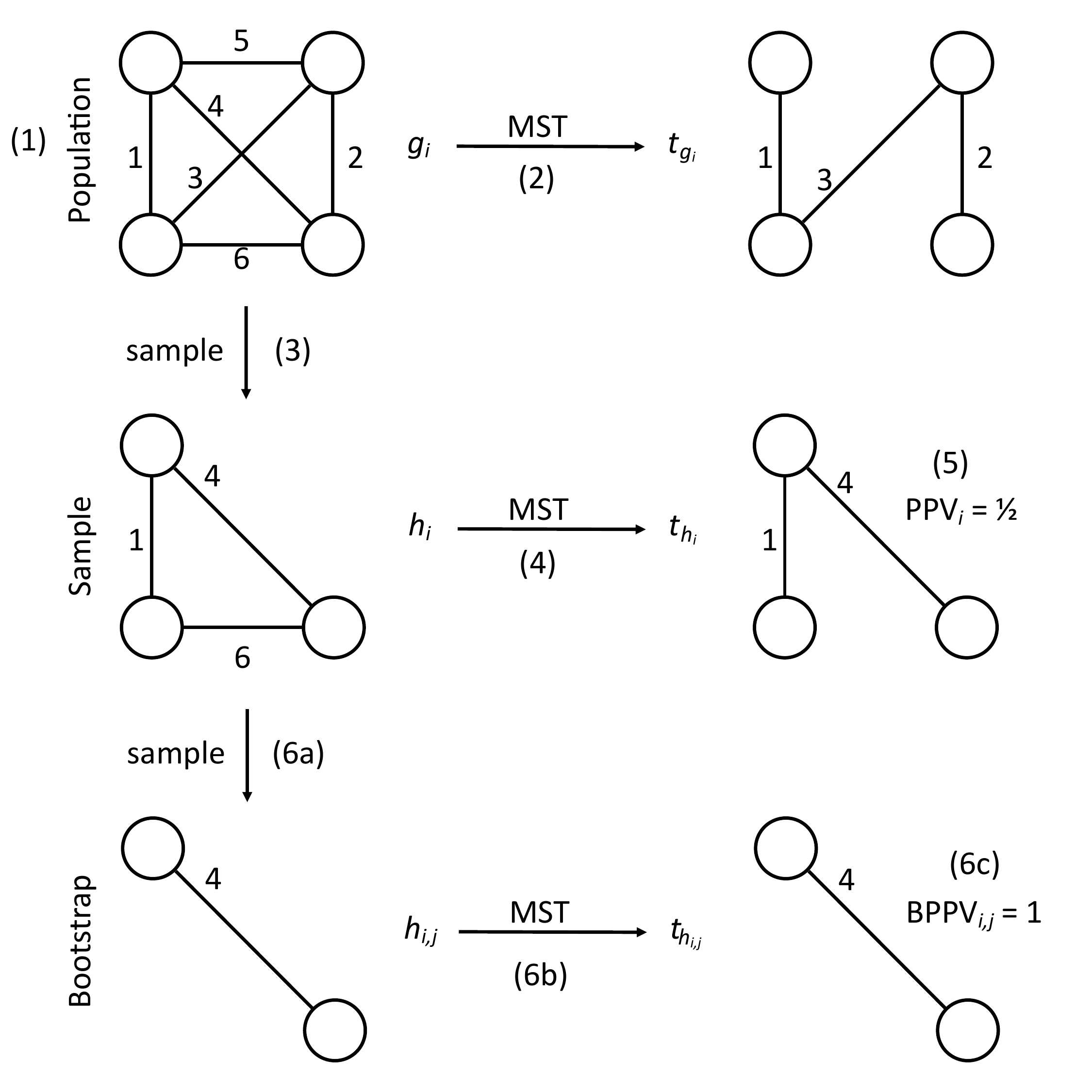}
\caption{Illustration of one replication.}
\label{figure:schematic}
\end{figure}

An entire simulation, with 1,000 replications, was run for each
of the following types of graphs (with $N = 100$ nodes):
\begin{enumerate}
\item Complete: Complete graph with weights uniformly distributed on $(0,1)$.
\item $G\left(N,\frac{1}{2}\right)$: First, a complete graph was generated
with weights uniformly distributed on $(0,1)$.
Then, each edge was included in the final graph with probability $\frac{1}{2}$.
\item Normal: Vertices were distributed in $\mathbb{R}^2$
according to a bivariate standard normal distribution,
and the weight of an edge connecting two vertices was equal to the Euclidean
distance between them.
\item Barab\'{a}si–Albert: Barab\'{a}si–Albert (BA) graph
with each new node attaching to three existing nodes and
with weights uniformly distributed on $(0,1)$.
\end{enumerate}
For each replication,
$n$ was set to $25$, $50$, and $75$,
and for each replication and each value of $n$,
the following types of sampling were used:
\begin{enumerate}
\item Uniform: Nodes were sampled uniformly at random.
\item Near: For complete graphs,
node $i$'s probability of being sampled was proportional to
$\max \{s_1,\dotsc,s_{N}\} - s_i + \min \{s_1,\dotsc,s_{N}\}$,
where $s_i$ is the total weight of all edges adjacent to node $i$.
This simulates preferentially selecting nodes that are close to other nodes,
while ensuring that every node has positive probability of being selected.
For non-complete graphs,
node $i$'s probability of being sampled was proportional to
$d_i$ or $d_i + 1$,
where $d_i$ is the degree of node $i$,
if the minimum degree was positive or zero,
respectively.
This simulates preferentially selecting nodes with many neighbors,
while ensuring that every node has positive probability of being selected.
\item Far: For complete graphs,
node $i$'s probability of being sampled was proportional to
$s_i$.
This simulates preferentially selecting nodes that are far from other nodes.
For non-complete graphs,
node $i$'s probability of being sampled was proportional to
$\max \{d_1,\dotsc,d_{N}\} - d_i + \max\{1,\min \{d_1,\dotsc,d_{N}\}\}$.
This simulates preferentially selecting nodes with few neighbors,
while ensuring that every node has positive probability of being selected.
\item Random Walk:
The following algorithm was repeated until $n$ nodes were recorded
in the vector $v$:
A node was selected uniformly at random 
from nodes not already in $v$ and recorded.
Suppose it was node $i$.
If node $i$ had no neighbors,
the process was restarted.
(Note that node $i$ is not thrown out if it has no neighbors;
even isolated nodes can be included in $v$.)
If node $i$ had neighbors,
with labels $j_1,\dotsc,j_{d_i}$,
one of these neighbors was selected at random to be the next recorded node.
The probabilities were not uniform;
for complete and non-complete graphs,
node $j_1$'s probability of being selected was proportional to
$\max \{s_{j_1},\dotsc,s_{j_{d_i}}\} - s_{j_1} + \min \{s_{j_1},\dotsc,s_{j_{d_i}}\}$.
This simulates preferentially selecting a neighbor that is close
to the current node,
while ensuring that every neighbor has positive probability of begin selected.
\end{enumerate}

One additional sampling method was used only for the ``normal'' graph.
For each replication,
all nodes in the first quadrant were sampled
(i.e., nodes with $x$ and $y$ coordinates greater than or equal to $0$);
then, all nodes in the first and second quadrant were sampled
(i.e., nodes with $x$ coordinate greater than or equal to $0$);
finally, all nodes in the first, second, and fourth quadrant were sampled
(i.e., nodes with $x$ or $y$ coordinate greater than or equal to $0$).
Note that for each replication,
approximately (but not necessarily exactly) $25$, $50$, and $75$ nodes
are sampled.
No bootstrapping was performed for this sampling method.

\subsection{HIV Genetic Distance Network}

Infectious disease researchers have used MSTs
to infer the transmission pathway of pathogens
\cite{campbell2017,spada2004}.
They typically sequence a portion of the pathogen's genome from
human tissue samples;
calculate distances between samples based on those sequences;
construct networks in which each sample is a node and two nodes
are connected by an edge if the distance between them is below a specified
cut-off;
weight each edge by the distance between the two nodes;
and then construct an MST from this weighted graph.
The MST is a natural starting point when trying to determine the transmission
pathway of the pathogen.
The researchers are usually only interested in
the first time a person was infected, so they want to eliminate cycles,
and the person most likely to have infected a given individual
is assumed to be
whoever has the pathogen with the most similar genetic makeup.

The Primary Infection Research Consortium at UC San Diego
(PIRC) \cite{le2013,morris2010}
provided an edgelist for an HIV genetic distance network.
Each year, the PIRC recruits up to 100 people who are newly diagnosed with HIV.
Both specimens and clinical data are collected upon recruitment
and then at regular intervals thereafter.
Participants with chronic HIV infection are followed for twelve weeks
and participants with acute HIV infection are followed for several years.

Each of the 1,234 nodes in the edgelist corresponded to an HIV sample;
edge weights were genetic distances
calculated using the HIV-TRACE method \cite{kosakovsky2018}.
This method aligns a sample sequence to a reference sequence
and then calculates distances between each pair of sample sequences.
As in \cite{campbell2017} and \cite{little2014},
edges with distances greater than 1.5\% were deleted.
As a result,
nodes that were greater than 1.5\% distance from all other nodes
became isolated, and were removed.
This yielded a graph with 588 nodes, 984 edges,
and 171 components.
Figure \ref{figure:components} displays the number of components of
each size.

\begin{figure}
\centering
\includegraphics[width = 0.8\linewidth]{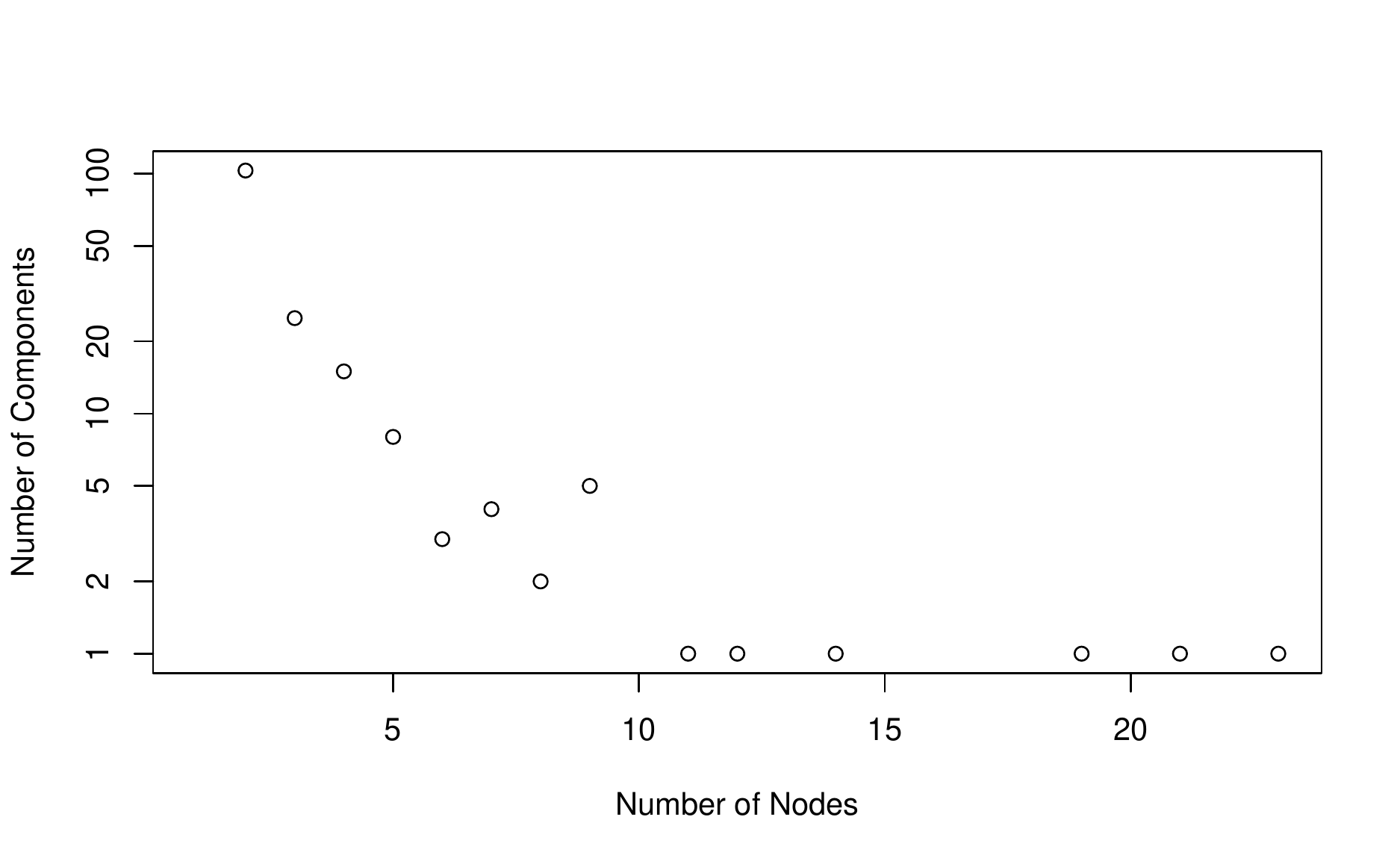}
\caption{The number of components of each size in the HIV graph.}
\label{figure:components}
\end{figure}

Fifty-three edges had a weight of 0.
Unfortunately, due to uncertainty in the distance estimation process
\cite{tamura1993},
even if the edges between samples A and B and between B and C
both had weights of 0,
the edge between samples A and C was not always 0.
Thus, weights of 0 were set to one-half the minimum positive edge weight.
Regarding the edge weights,
746 were unique,
69 were shared by two edges,
11 were shared by three edges,
two were shared by four edges,
one was shared by six edges,
and one was shared by fifty-three edges.
Figure \ref{figure:hist} displays the number of edges
that had each edge weight.
The minimum positive difference between any two edge weights was $10^{-9}$.

\begin{figure}
\centering
\includegraphics[width = 0.8\linewidth]{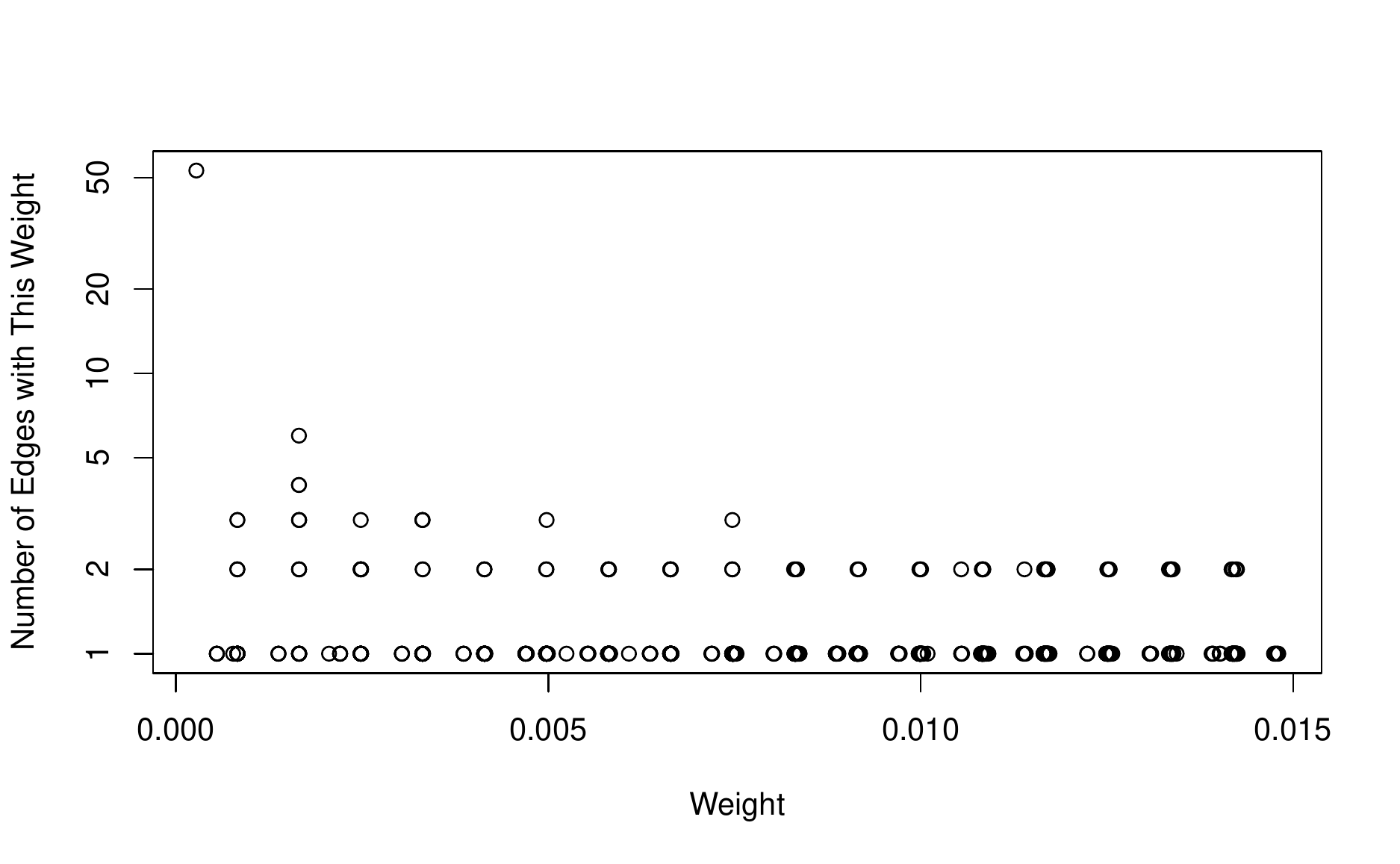}
\caption{Number of edges with each edge weight in the HIV graph.}
\label{figure:hist}
\end{figure}

In order to ensure unique edge weights,
one of the orderings described in the Theory section
was selected uniformly at random
and used throughout analysis.
In other words,
the edges were ordered from smallest to greatest weight,
with ties broken arbitrarily.

The same algorithm described in the Simulation Study subsection
was used to analyze the empirical data,
with the following modifications:
\begin{itemize}
\item For each of the 1,000 replications,
$g_i = g$ was the empirical HIV genetic distance network.
That is, the same graph was used each time.
\item In order to sample the same proportion of nodes as in the simulation
study (25\%, 50\%, and 75\%),
for each replication,
$n$ was set to $147$, $294$, and $441$.
\end{itemize}
Figure \ref{figure:hiv} displays
the entire HIV genetic distance network, its MST,
a subgraph induced by sampling 50\% of the nodes uniformly at random,
and the MST of the induced subgraph.

\begin{figure}
\centering
\includegraphics[width = 0.8\linewidth]{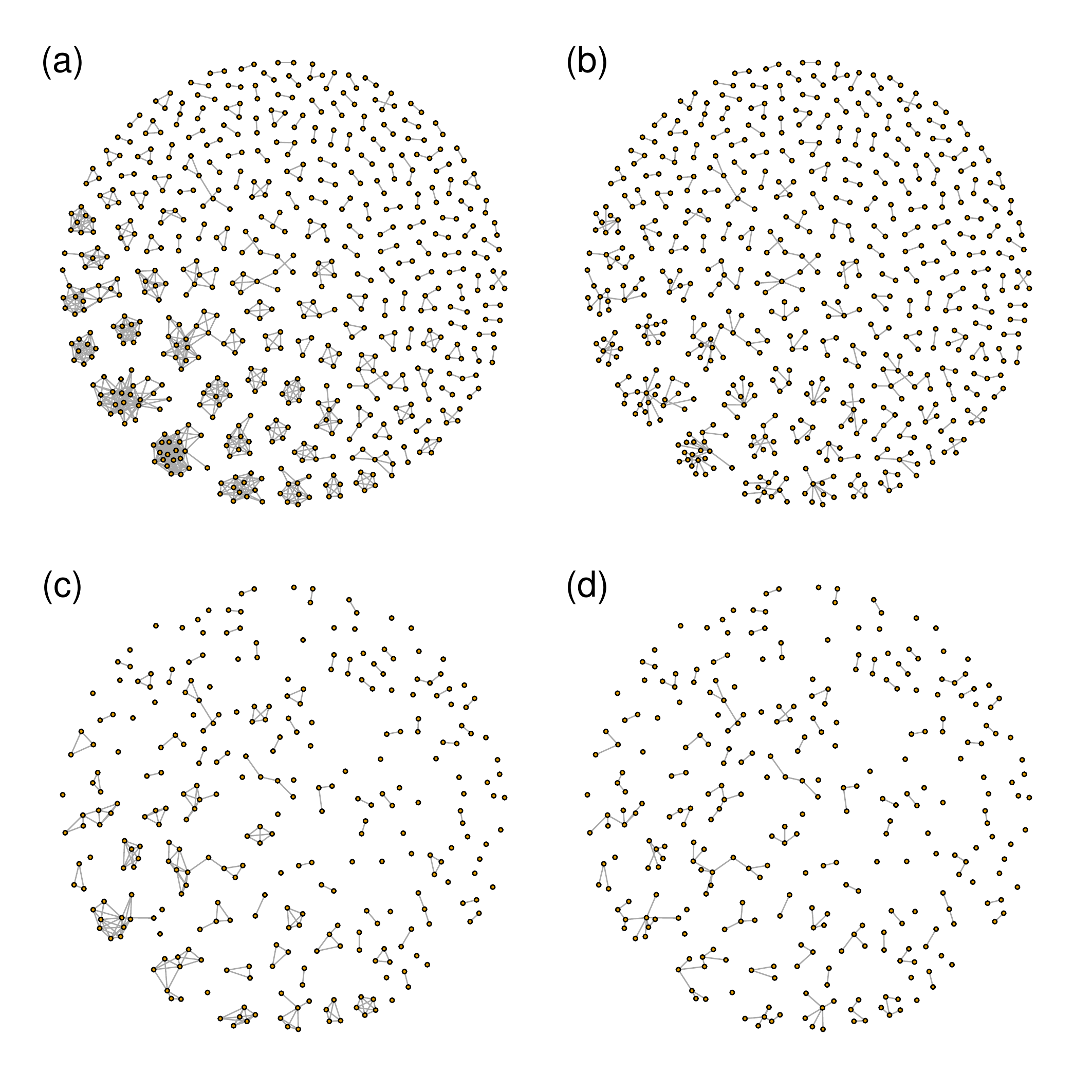}
\caption{(a) The HIV genetic distance network.
It has 588 nodes, 984 edges, and 171 components.
(b) The population MST of the HIV genetic distance network.
(c) A subgraph induced by sampling 50\% of the nodes in the HIV genetic distance network uniformly at random.
(d) The sample MST of the induced subgraph.
For this sample,
the PPV is 0.737.}
\label{figure:hiv}
\end{figure}

The PIRC also provided three-digit zip codes
for 564 of the 588 nodes in the graph.
For each of the three zip codes with the most nodes,
an MST was created using only nodes from that zip code,
and the proportion of edges in each MST that were also in the population MST
was calculated.
These three zip codes accounted for 546 nodes,
or 92.9\% of the nodes in the graph.
The next-most-represented zip code had only seven nodes,
or 1.2\% of the graph.

All graphs were undirected.
All simulations were run in R version 3.6.1,
using the package igraph \cite{csardi2006},
on the O2 High Performance Compute Cluster,
supported by the Research Computing Group, at Harvard Medical School.
See http://rc.hms.harvard.edu for more information.
The package igraph uses Prim's algorithm
\cite{prim1957} to find the MST.
Code is available at https://github.com/onnela-lab/mst.

The PIRC was approved by the University of 
California--San Diego's Human Research Protection Program
(Project \#140585 and Project \#191088).
The current study was determined to be not human subjects research
by the IRB of the Harvard T.H. Chan School of Public Health
(Protocol \# IRB19-2166).

\section{Results}

Results for the simulation study and empirical data
are in Tables \ref{table:theta} and \ref{table:zip}.
The results for quadrant sampling of normal graphs
are as follows:
for $n = 25$, $\overline{\text{PPV}} = 0.246$
(95\% CI 0.240-0.252);
for $n = 50$, $\overline{\text{PPV}} = 0.497$
(95\% CI 0.492-0.503);
and for $n = 75$, $\overline{\text{PPV}} = 0.747$
(95\% CI 0.743-0.752).

\newgeometry{left=0.5in,right=0.5in,bottom=0.5in,top=0.5in}
%\begin{figure*}
\begin{table}
\centering
\begin{tabular}{cllcccc}
\toprule
 & & & \multicolumn{4}{c}{\textbf{Type of Sampling}} \\
\textbf{Graph} & $\bm{n}$ & \textbf{Statistic} & \textbf{Uniform} & \textbf{Near} & \textbf{Far} & \textbf{Random Walk} \\
\midrule
\multirow{9}{*}{\begin{turn}{90}
Complete
\end{turn}} & $25$ & $\overline{\text{PPV}}$ & 0.245 (0.240-0.251) & 0.255 (0.249-0.261) & 0.247 (0.241-0.252) & 0.268 (0.262-0.274) \\
& & $\overline{\overline{\text{BPPV}}}$ & 0.240 (0.239-0.241) & 0.251 (0.250-0.252) & 0.229 (0.228-0.230) & 0.300 (0.298-0.301) \\
& & $\overline{\text{AUC}}$ & 0.862 (0.858-0.865) & 0.871 (0.868-0.874) & 0.842 (0.838-0.845) & 0.908 (0.906-0.910) \\
& $50$ & $\overline{\text{PPV}}$ & 0.503 (0.499-0.508) & 0.508 (0.504-0.513) & 0.494 (0.490-0.499) & 0.513 (0.509-0.518) \\
& & $\overline{\overline{\text{BPPV}}}$ & 0.500 (0.500-0.501) & 0.510 (0.509-0.510) & 0.491 (0.491-0.492) & 0.523 (0.522-0.523) \\
& & $\overline{\text{AUC}}$ & 0.983 (0.983-0.983) & 0.984 (0.983-0.984) & 0.981 (0.981-0.982) & 0.984 (0.984-0.984) \\
& $75$ & $\overline{\text{PPV}}$ & 0.748 (0.745-0.751) & 0.754 (0.751-0.757) & 0.744 (0.740-0.747) & 0.761 (0.758-0.764) \\
& & $\overline{\overline{\text{BPPV}}}$ & 0.747 (0.746-0.747) & 0.753 (0.753-0.754) & 0.740 (0.740-0.741) & 0.756 (0.756-0.757) \\
& & $\overline{\text{AUC}}$ & 0.996 (0.996-0.996) & 0.996 (0.996-0.997) & 0.996 (0.996-0.996) & 0.997 (0.996-0.997) \\
\hline
\multirow{9}{*}{\begin{turn}{90}
$G(n,p)$
\end{turn}} & $25$ & $\overline{\text{PPV}}$ & 0.255 (0.250-0.261) & 0.254 (0.248-0.260) & 0.245 (0.239-0.250) & 0.298 (0.293-0.304) \\
& & $\overline{\overline{\text{BPPV}}}$ & 0.249 (0.248-0.251) & 0.258 (0.257-0.259) & 0.242 (0.241-0.244) & 0.397 (0.395-0.399) \\
& & $\overline{\text{AUC}}$ & 0.728 (0.722-0.734) & 0.744 (0.738-0.750) & 0.683 (0.676-0.690) & 0.892 (0.889-0.894) \\
& $50$ & $\overline{\text{PPV}}$ & 0.497 (0.493-0.502) & 0.504 (0.500-0.509) & 0.495 (0.491-0.500) & 0.538 (0.533-0.542) \\
& & $\overline{\overline{\text{BPPV}}}$ & 0.500 (0.499-0.500) & 0.511 (0.510-0.511) & 0.490 (0.489-0.490) & 0.564 (0.563-0.564) \\
& & $\overline{\text{AUC}}$ & 0.965 (0.964-0.965) & 0.965 (0.964-0.965) & 0.959 (0.959-0.960) & 0.973 (0.972-0.973) \\
& $75$ & $\overline{\text{PPV}}$ & 0.752 (0.748-0.755) & 0.756 (0.753-0.759) & 0.742 (0.738-0.745) & 0.772 (0.769-0.776) \\
& & $\overline{\overline{\text{BPPV}}}$ & 0.747 (0.746-0.747) & 0.754 (0.754-0.755) & 0.740 (0.739-0.740) & 0.774 (0.774-0.775) \\
& & $\overline{\text{AUC}}$ & 0.993 (0.992-0.993) & 0.993 (0.992-0.993) & 0.992 (0.992-0.992) & 0.993 (0.993-0.994) \\
\hline
\multirow{9}{*}{\begin{turn}{90}
Normal
\end{turn}} & $25$ & $\overline{\text{PPV}}$ & 0.246 (0.241-0.252) & 0.254 (0.249-0.260) & 0.251 (0.245-0.256) & 0.267 (0.261-0.272) \\
& & $\overline{\overline{\text{BPPV}}}$ & 0.240 (0.239-0.241) & 0.251 (0.249-0.252) & 0.228 (0.227-0.230) & 0.300 (0.299-0.302) \\
& & $\overline{\text{AUC}}$ & 0.864 (0.861-0.868) & 0.873 (0.870-0.876) & 0.842 (0.839-0.846) & 0.906 (0.904-0.909) \\
& $50$ & $\overline{\text{PPV}}$ & 0.497 (0.492-0.502) & 0.503 (0.498-0.507) & 0.494 (0.489-0.498) & 0.512 (0.508-0.517) \\
& & $\overline{\overline{\text{BPPV}}}$ & 0.500 (0.499-0.500) & 0.509 (0.509-0.510) & 0.491 (0.490-0.491) & 0.522 (0.522-0.523) \\
& & $\overline{\text{AUC}}$ & 0.983 (0.983-0.983) & 0.983 (0.983-0.984) & 0.981 (0.981-0.981) & 0.984 (0.984-0.984) \\
& $75$ & $\overline{\text{PPV}}$ & 0.748 (0.744-0.751) & 0.755 (0.751-0.758) & 0.745 (0.741-0.748) & 0.758 (0.755-0.762) \\
& & $\overline{\overline{\text{BPPV}}}$ & 0.747 (0.746-0.747) & 0.753 (0.753-0.754) & 0.740 (0.740-0.741) & 0.756 (0.756-0.757) \\
& & $\overline{\text{AUC}}$ & 0.996 (0.996-0.996) & 0.996 (0.996-0.997) & 0.996 (0.996-0.996) & 0.997 (0.996-0.997) \\
\hline
\multirow{9}{*}{\begin{turn}{90}
Barab\'{a}si–Albert
\end{turn}} & $25$ & $\overline{\text{PPV}}$ & 0.389 (0.381-0.396) & 0.459 (0.453-0.466) & 0.377 (0.368-0.385) & 0.704 (0.698-0.709) \\
& & $\overline{\overline{\text{BPPV}}}$ & 0.905 (0.899-0.910) & 0.639 (0.633-0.646) & 0.974 (0.971-0.977) & 0.746 (0.743-0.749) \\
& & $\overline{\text{AUC}}$ & 0.510 (0.500-0.520) & 0.511 (0.504-0.518) & 0.523 (0.511-0.534) & 0.872 (0.869-0.874) \\
& $50$ & $\overline{\text{PPV}}$ & 0.538 (0.532-0.543) & 0.701 (0.697-0.706) & 0.465 (0.460-0.470) & 0.852 (0.849-0.855) \\
& & $\overline{\overline{\text{BPPV}}}$ & 0.722 (0.718-0.726) & 0.734 (0.732-0.736) & 0.830 (0.825-0.835) & 0.838 (0.837-0.839) \\
& & $\overline{\text{AUC}}$ & 0.684 (0.679-0.689) & 0.849 (0.847-0.852) & 0.554 (0.549-0.559) & 0.946 (0.945-0.947) \\
& $75$ & $\overline{\text{PPV}}$ & 0.755 (0.751-0.759) & 0.880 (0.878-0.883) & 0.649 (0.645-0.654) & 0.943 (0.941-0.945) \\
& & $\overline{\overline{\text{BPPV}}}$ & 0.775 (0.775-0.776) & 0.888 (0.887-0.889) & 0.707 (0.705-0.709) & 0.932 (0.932-0.933) \\
& & $\overline{\text{AUC}}$ & 0.912 (0.910-0.915) & 0.969 (0.968-0.969) & 0.719 (0.715-0.723) & 0.984 (0.984-0.985) \\
\hline
\multirow{9}{*}{\begin{turn}{90}
HIV Network
\end{turn}} & $147$ & $\overline{\text{PPV}}$ & 0.578 (0.572-0.583) & 0.688 (0.684-0.693) & 0.635 (0.630-0.641) & 0.990 (0.989-0.991) \\
& & $\overline{\overline{\text{BPPV}}}$ & 0.800 (0.796-0.804) & 0.653 (0.651-0.655) & 0.937 (0.935-0.940) & 0.975 (0.974-0.975) \\
& & $\overline{\text{AUC}}$ & 0.598 (0.593-0.603) & 0.786 (0.784-0.789) & 0.626 (0.620-0.632) & 0.997 (0.997-0.998) \\
& $294$ & $\overline{\text{PPV}}$ & 0.727 (0.724-0.730) & 0.888 (0.886-0.890) & 0.729 (0.726-0.731) & 0.995 (0.995-0.995) \\
& & $\overline{\overline{\text{BPPV}}}$ & 0.795 (0.794-0.796) & 0.847 (0.847-0.848) & 0.871 (0.870-0.873) & 0.992 (0.992-0.992) \\
& & $\overline{\text{AUC}}$ & 0.855 (0.853-0.857) & 0.963 (0.963-0.964) & 0.788 (0.786-0.790) & 0.999 (0.998-0.999) \\
& $441$ & $\overline{\text{PPV}}$ & 0.868 (0.866-0.869) & 0.972 (0.972-0.973) & 0.836 (0.834-0.838) & 0.998 (0.998-0.998) \\
& & $\overline{\overline{\text{BPPV}}}$ & 0.880 (0.879-0.880) & 0.960 (0.959-0.960) & 0.872 (0.871-0.873) & 0.997 (0.997-0.997) \\
& & $\overline{\text{AUC}}$ & 0.944 (0.943-0.944) & 0.997 (0.997-0.997) & 0.919 (0.919-0.920) & 0.999 (0.999-0.999) \\
\bottomrule
\end{tabular}
\caption{\label{table:theta}
Results for complete, $G(n,p)$,
normal (excluding quadrant sampling),
Barab\'{a}si–Albert,
and HIV genetic distance network graphs.
Data are presented as mean (95\% CI).
Each simulated graph had 100 nodes.
The HIV genetic distance network is from
the Primary Infection Resource Consortium \cite{le2013,morris2010}.
PPV = positive predictive value;
BPPV = bootstrap positive predictive value;
AUC = area under the receiver operating characteristic curve.}
\end{table}
%\end{figure*}

\newgeometry{left=1in,right=1in,bottom=1in,top=1in}

\begin{table}
\centering
\begin{tabular}{rrrr}
\toprule
\multicolumn{1}{c}{\textbf{Zip Code}} & \multicolumn{1}{c}{\textbf{\# of Nodes}} & \multicolumn{1}{c}{\textbf{\% of Total Nodes}} & \multicolumn{1}{c}{\textbf{\% of Sample MST Edges in Population MST}} \\
\midrule
921 & 433 & 73.6\% & 87.5\% \\
920 & 63 & 10.7\% & 92.9\% \\
919 & 50 & 8.5\% & 44.4\% \\
\bottomrule
\end{tabular}
\caption{\label{table:zip}
Results from creating MSTs from nodes belonging to each
of the most-represented zip codes in the data.}
\end{table}

\doublespacing

For complete, $G(n,p)$, and normal graphs,
$\overline{\text{PPV}} \approx \frac{n}{N}$
when nodes are sampled uniformly at random;
$\overline{\text{PPV}} > \frac{n}{N}$
when nodes that have many neighbors or low total edge weight are
preferentially sampled,
or when an edge-weighted random walk is used to sample nodes; and
$\overline{\text{PPV}} < \frac{n}{N}$
when nodes that have few neighbors or high total edge weight are
preferentially sampled.
For normal graphs,
$\overline{\text{PPV}} \approx \frac{n}{N}$
when nodes are sampled by quadrant.

For BA graphs and uniform sampling,
$\overline{\text{PPV}} > \frac{n}{N}$.
``Near'' sampling increases $\overline{\text{PPV}}$
whereas ``far'' sampling decreases it.
The random walk produces the highest values of $\overline{\text{PPV}}$
of any of the sampling methods.
The simulations using the PIRC data have similar results to the BA graphs
but with even higher values of $\overline{\text{PPV}}$.

Across graph types and sampling methods,
$\overline{\overline{\text{BPPV}}}$ does not have
a consistent relationship with
$\overline{\text{PPV}}$.
Sometimes they have overlapping confidence intervals,
sometimes $\overline{\overline{\text{BPPV}}} > \overline{\text{PPV}}$,
and sometimes $\overline{\overline{\text{BPPV}}} < \overline{\text{PPV}}$;
it is hard to generalize about when each scenario arises.
That said,
$\overline{\overline{\text{BPPV}}}$ is closer to
$\overline{\text{PPV}}$ at higher sample sizes,
for all graphs and all types of sampling.

For complete, $G(n,p)$, and normal graphs,
$\overline{\text{AUC}}$ is almost always above $0.75$,
with many values above $0.90$.
For BA graphs and the PIRC graph,
$\overline{\text{AUC}}$ is still always above $0.50$,
with many values above $0.90$.

When sampling by zip code,
the percentage of sample MST edges that are also in the population MST
is far greater than the percentage of nodes sampled.

\section{Discussion}
\label{sec:discussion}

In spite of the wide use of the MST on sample networks,
little was known about what could be inferred from it
about the MST of population networks.
This study examined exactly that.
Returning to the questions posed in the Introduction,
we can say the following:
\begin{enumerate}
\item Given that an edge is in the sample graph
but not the sample MST,
what is the probability that it is not in the population MST?

This probability is $1$,
regardless of the type of graph or sampling method,
provided the edge weights are unique.
\item Given that an edge appears in the sample MST,
what is the probability that it appears in the population MST?

This depends on the number of nodes sampled,
the type of graph, and the type of sampling.
This conditional probability is maximized by increasing the sample size;
starting with an underlying BA graph;
and either preferentially sampling nodes that are ``near'' other nodes
or using an edge-weighted random walk.
Of course, applied researchers will not be able to choose their underlying
graph type or tell which nodes
are high degree or have low total edge weight before sampling.
Thus, an edge-weighted random walk may be their best bet.
The probability that an edge appears in the population MST given that
it is in the sample MST is minimized by decreasing $n$;
starting with an underlying complete, $G(n,p)$, or normal graph;
and preferentially sampling nodes that are ``far'' from other nodes.

Fortunately, applied researchers may already be using the edge-weighted
random walk.
In that sampling method,
the neighbor of a selected node is more likely to be
selected as well if their viral genome is closer to the first node.
In the real world,
this is achieved by contact tracing or partner notification.
For example, if someone tests positive for HIV,
efforts are made to test anyone they recently had sexual contact with
or shared a needle with.
The idea is to find those individuals who either transmitted the pathogen
to the original patient or who contracted the disease from the original
patient.
Both of these groups of people are likely to be carrying pathogens
that are genetically similar to the pathogens carried by the original patient.
\item Can this probability be estimated from the sample graph
using bootstrapping?

This depends on the number of nodes sampled,
the type of graph, and the sampling method.
No general recommendation can be made, unfortunately.
\item Can we use bootstrapping to increase our chances of identifying
edges in the population MST?

There is a strong relationship between the number of times
a sampled edge appears in a bootstrapped MST and whether or not
it is in the population MST.
However, it is as yet unclear how to capitalize on this relationship.
More research is needed.
\end{enumerate}

The results for complete, $G(n,p)$ and normal graphs were very similar
to each other and differed from the results for BA graphs.
This makes sense because sampling nodes uniformly
from each of the first three graph types
yields a graph of a similar type;
in contrast, sampling nodes uniformly from a BA graph
does not yield a BA graph
\cite{stumpf2005}.
The results for the $G(n,p)$ graph differed somewhat from the results
for the complete and normal graph when a random walk was used to sample nodes.
For this sampling method, the average PPV was higher for $G(n,p)$
than for complete and normal graphs.
This may indicate that an edge-weighted random walk
leads to an increased PPV when the underlying graph is not complete,
i.e., when not all possible edges are present.
The results for the PIRC data were similar to the results for the BA graphs.
This makes sense given the similarity in degree distribution
(see Figure \ref{figure:dd}).

\begin{figure}
\centering
\includegraphics[width = 0.8\linewidth]{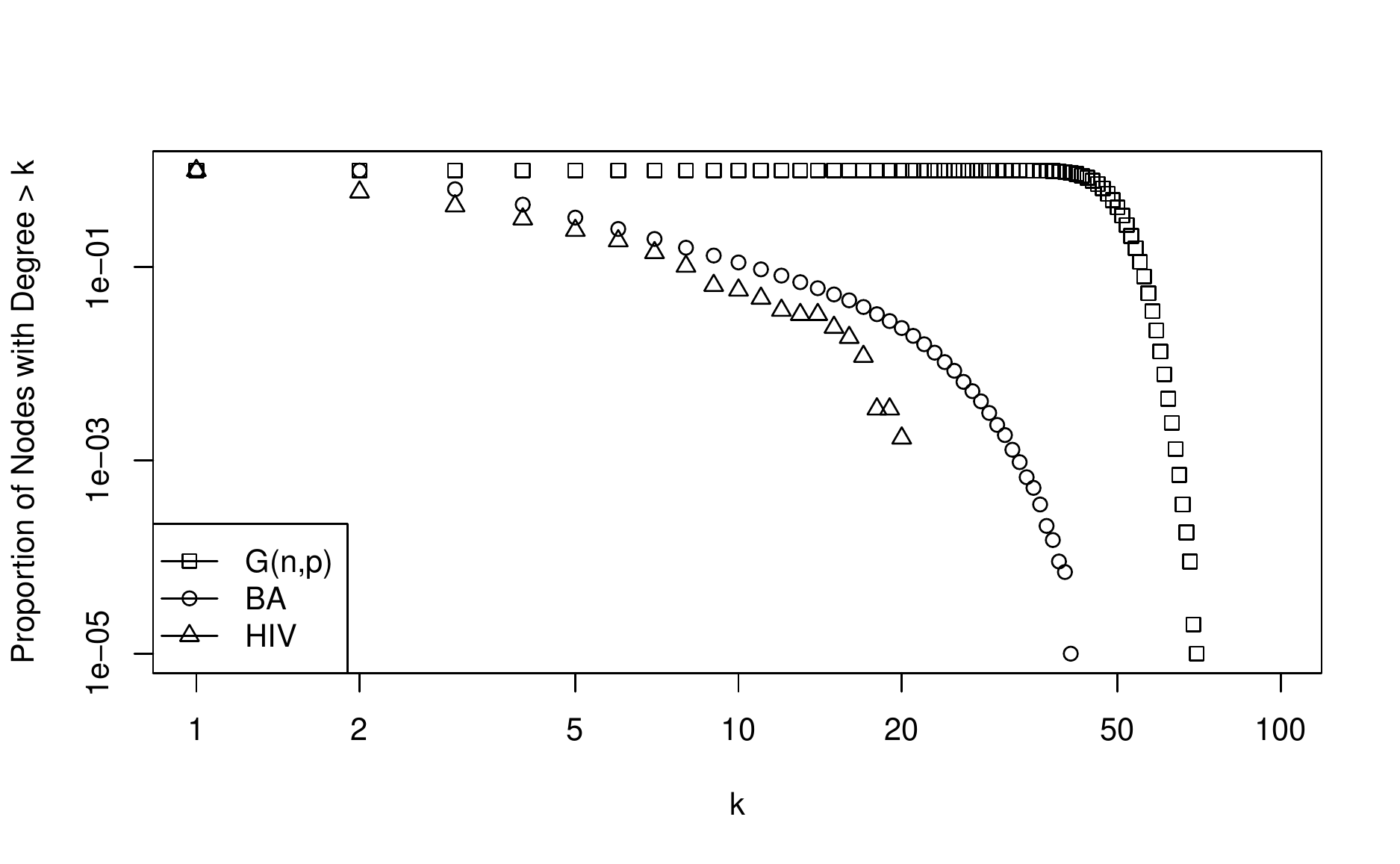}
\caption{Degree distributions for the
$G(n,p)$, Barab\'{a}si–Albert, and PIRC graphs.
Values for the $G(n,p)$
and Barab\'{a}si–Albert graphs are averages across 1,000 replications.
Values for the complete and normal graphs are not shown
because in those graphs, each node has degree $n - 1$.}
\label{figure:dd}
\end{figure}

When sampling by zip code,
the proportion of edges in the sample MST that are also in the population MST
is much, much higher than $\frac{n}{N}$.
This is good news for applied researchers,
because it implies that sampling can be limited to a single geographic area
and still identify most of the edges that are in the population MST.
It is interesting to note that this contrasts with the location-based
sampling that was performed with the simulated normal graphs.
With the normal graphs, sampling by quadrant yielded
conditional probabilities approximately equal to $\frac{n}{N}$.

One limitation of this study is that the PIRC data have non-unique
edge weights,
meaning the MST may not be unique.
Future studies can examine whether the number of times an edge
appears in sample MSTs is indicative of the number of times
it appears in the population MSTs.
Further research could also examine the impact of measuring edge weights
with error.

\section{Declarations}

\subsection{Acknowledgments}

The authors would like to thank Susan Little,
Christy Anderson, Martin Furey, Felix Torres,
Sergei Kosakovsky Pond, and the Primary Infection Resource Consortium
for sharing and explaining the empirical data.
The authors would also like to thank Rui Wang, Alessandro Vespignani, and
Edoardo Airoldi for their helpful suggestions.

\subsection{Funding}

Jonathan Larson is supported by NIH T32 AI007358.
Jukka-Pekka Onnela is supported by NIAID R01 AI138901.

\subsection{Affiliations}

Jonathan Larson is a student and Jukka-Pekka Onnela is an Associate Professor
in the Department of Biostatistics
at Harvard T.H. Chan School of Public Health.

\subsection{Author Contributions}

J.L. designed the research, performed the research,
and analyzed the data.
J.P.O. supervised the research.
J.L. and J.P.O. wrote the paper.

\subsection{Competing Interests}

The authors declare no competing interests.

\bibliographystyle{vancouver}
\bibliography{bib}

\begin{thebibliography}{10}

\bibitem{nesetril2001}
Ne\v{s}et\v{r}il J, Milkov\'{a} E, Ne\v{s}et\v{r}ilov\'{a} H.
\newblock Otakar Bor\r{u}vka on minimum spanning tree problem Translation of
  both the 1926 papers, comments, history.
\newblock Discrete Mathematics. 2001;233(1):3--36.

\bibitem{prim1957}
Prim RC.
\newblock Shortest Connection Networks And Some Generalizations.
\newblock Bell System Technical Journal. 1957;36(6):1389--1401.
\newblock Available from:
  \url{https://onlinelibrary.wiley.com/doi/abs/10.1002/j.1538-7305.1957.tb01515.x}.

\bibitem{kruskal1956}
Kruskal JB.
\newblock On the Shortest Spanning Subtree of a Graph and the Traveling
  Salesman Problem.
\newblock Proceedings of the American Mathematical Society. 1956;7(1):48--50.
\newblock Available from: \url{http://www.jstor.org/stable/2033241}.

\bibitem{tewarie2015}
Tewarie P, {van Dellen} E, Hillebrand A, Stam CJ.
\newblock The minimum spanning tree: An unbiased method for brain network
  analysis.
\newblock NeuroImage. 2015;104:177--188.
\newblock Available from:
  \url{https://www.sciencedirect.com/science/article/pii/S1053811914008398}.

\bibitem{vandellen2018}
van Dellen E, Sommer IE, Bohlken MM, Tewarie P, Draaisma L, Zalesky A, et~al.
\newblock Minimum spanning tree analysis of the human connectome.
\newblock Human Brain Mapping. 2018;39(6):2455--2471.
\newblock Available from:
  \url{https://onlinelibrary.wiley.com/doi/abs/10.1002/hbm.24014}.

\bibitem{wang2020}
{Wang} B, {Chen} Y, {Liu} W, {Qin} J, {Du} Y, {Han} G, et~al.
\newblock Real-time hierarchical supervoxel segmentation via a minimum spanning
  tree.
\newblock IEEE Transactions on Image Processing. 2020;29:9665--9677.

\bibitem{jin2020}
Jin Y, Zhao H, Gu F, Bu P, Na M.
\newblock A spatial minimum spanning tree filter.
\newblock Measurement Science and Technology. 2020 jan;32(1):015204.
\newblock Available from: \url{https://doi.org/10.1088/1361-6501/abaa65}.

\bibitem{wu2018}
Wu B, Yu B, Wu Q, Chen Z, Yao S, Huang Y, et~al.
\newblock An extended minimum spanning tree method for characterizing local
  urban patterns.
\newblock International Journal of Geographical Information Science.
  2018;32(3):450--475.

\bibitem{mantegna1999}
Mantegna RN.
\newblock Hierarchical structure in financial markets.
\newblock The European Physical Journal B - Condensed Matter and Complex
  Systems. 1999;11(1):193--197.

\bibitem{onnela2002}
Onnela JP, Chakraborti A, Kaski K, Kerti\'{e}sz J.
\newblock Dynamic asset trees and portfolio analysis.
\newblock The European Physical Journal B - Condensed Matter and Complex
  Systems. 2002;30(3):285--288.

\bibitem{onnela2003}
Onnela JP, Chakraborti A, Kaski K, Kert\'esz J, Kanto A.
\newblock Dynamics of market correlations: Taxonomy and portfolio analysis.
\newblock Physical Review E. 2003 Nov;68:056110.
\newblock Available from:
  \url{https://link.aps.org/doi/10.1103/PhysRevE.68.056110}.

\bibitem{li2020}
Li K, Zhang S, Song X, Weyrich A, Wang Y, Liu X, et~al.
\newblock Genome evolution of blind subterranean mole rats: Adaptive peripatric
  versus sympatric speciation.
\newblock Proceedings of the National Academy of Sciences.
  2020;117(51):32499--32508.
\newblock Available from: \url{https://www.pnas.org/content/117/51/32499}.

\bibitem{steinbrenner2020}
Steinbrenner AD, Mu{\~n}oz-Amatria{\'\i}n M, Chaparro AF, Aguilar-Venegas JM,
  Lo S, Okuda S, et~al.
\newblock A receptor-like protein mediates plant immune responses to
  herbivore-associated molecular patterns.
\newblock Proceedings of the National Academy of Sciences.
  2020;117(49):31510--31518.
\newblock Available from: \url{https://www.pnas.org/content/117/49/31510}.

\bibitem{manning2020}
Manning CD, Clark K, Hewitt J, Khandelwal U, Levy O.
\newblock Emergent linguistic structure in artificial neural networks trained
  by self-supervision.
\newblock Proceedings of the National Academy of Sciences.
  2020;117(48):30046--30054.
\newblock Available from: \url{https://www.pnas.org/content/117/48/30046}.

\bibitem{saul2020}
Saul LK.
\newblock A tractable latent variable model for nonlinear dimensionality
  reduction.
\newblock Proceedings of the National Academy of Sciences.
  2020;117(27):15403--15408.
\newblock Available from: \url{https://www.pnas.org/content/117/27/15403}.

\bibitem{matsumura2020}
Matsumura H, Hsiao MC, Lin YP, Toyoda A, Taniai N, Tarora K, et~al.
\newblock Long-read bitter gourd (Momordica charantia) genome and the genomic
  architecture of nonclassic domestication.
\newblock Proceedings of the National Academy of Sciences.
  2020;117(25):14543--14551.
\newblock Available from: \url{https://www.pnas.org/content/117/25/14543}.

\bibitem{hahn2020}
Hahn M, Jurafsky D, Futrell R.
\newblock Universals of word order reflect optimization of grammars for
  efficient communication.
\newblock Proceedings of the National Academy of Sciences.
  2020;117(5):2347--2353.
\newblock Available from: \url{https://www.pnas.org/content/117/5/2347}.

\bibitem{bertsimas1990}
Bertsimas DJ.
\newblock The probabilistic minimum spanning tree problem.
\newblock Networks. 1990;20:245--275.

\bibitem{goemans2006}
Goemans MX, Vondr\'{a}k J.
\newblock Covering minimum spanning trees of random subgraphs.
\newblock Random Structures \& Algorithms. 2006;29(3):257--276.
\newblock Available from:
  \url{https://onlinelibrary.wiley.com/doi/abs/10.1002/rsa.20115}.

\bibitem{torkestani2012}
Torkestani JA, Meybodi MR.
\newblock A learning automata-based heuristic algorithm for solving the minimum
  spanning tree problem in stochastic graphs.
\newblock The Journal of Supercomputing. 2012;59:1035–1054.

\bibitem{raphael2019}
(https://cs~stackexchange com/users/98/raphael) R. Do the minimum spanning
  trees of a weighted graph have the same number of edges with a given
  weight?;.
\newblock URL:https://cs.stackexchange.com/q/2211 (version: 2019-05-21).
\newblock Computer Science Stack Exchange.
\newblock Available from: \url{https://cs.stackexchange.com/q/2211}.

\bibitem{campbell2017}
Campbell EM, Jia H, Shankar A, Hanson D, Luo W, Masciotra S, et~al.
\newblock Detailed Transmission Network Analysis of a Large Opiate-Driven
  Outbreak of HIV Infection in the United States.
\newblock The Journal of Infectious Diseases. 2017 10;216(9):1053--1062.
\newblock Available from: \url{https://doi.org/10.1093/infdis/jix307}.

\bibitem{spada2004}
Spada E, Sagliocca L, Sourdis J, Garbuglia AR, Poggi V, De~Fusco C, et~al.
\newblock Use of the Minimum Spanning Tree Model for Molecular Epidemiological
  Investigation of a Nosocomial Outbreak of Hepatitis C Virus Infection.
\newblock Journal of Clinical Microbiology. 2004;42(9):4230--4236.
\newblock Available from: \url{https://jcm.asm.org/content/42/9/4230}.

\bibitem{le2013}
Le T, Wright EJ, Smith DM, He W, Catano G, Okulicz JF, et~al.
\newblock Enhanced CD4+ T-Cell Recovery with Earlier HIV-1 Antiretroviral
  Therapy.
\newblock New England Journal of Medicine. 2013;368(3):218--230.
\newblock PMID: 23323898.
\newblock Available from: \url{https://doi.org/10.1056/NEJMoa1110187}.

\bibitem{morris2010}
Morris SR, Little SJ, Cunningham T, Garfein RS, Richman DD, Smith DM.
\newblock Evaluation of an HIV Nucleic Acid Testing Program With Automated
  Internet and Voicemail Systems to Deliver Results.
\newblock Annals of Internal Medicine. 2010;152(12):778--785.
\newblock PMID: 20547906.
\newblock Available from:
  \url{https://www.acpjournals.org/doi/abs/10.7326/0003-4819-152-12-201006150-00005}.

\bibitem{kosakovsky2018}
Kosakovsky~Pond SL, Weaver S, Leigh~Brown AJ, Wertheim JO.
\newblock {HIV-TRACE (TRAnsmission Cluster Engine): a Tool for Large Scale
  Molecular Epidemiology of HIV-1 and Other Rapidly Evolving Pathogens}.
\newblock Molecular Biology and Evolution. 2018 01;35(7):1812--1819.
\newblock Available from: \url{https://doi.org/10.1093/molbev/msy016}.

\bibitem{little2014}
Little SJ, Kosakovsky~Pond SL, Anderson CM, Young JA, Wertheim JO, Mehta SR,
  et~al.
\newblock Using HIV Networks to Inform Real Time Prevention Interventions.
\newblock PLOS ONE. 2014 06;9(6):1--8.
\newblock Available from: \url{https://doi.org/10.1371/journal.pone.0098443}.

\bibitem{tamura1993}
Tamura K, Nei M.
\newblock Estimation of the number of nucleotide substitutions in the control
  region of mitochondrial DNA in humans and chimpanzees.
\newblock Molecular Biology and Evolution. 1993 05;10(3):512--526.
\newblock Available from:
  \url{https://doi.org/10.1093/oxfordjournals.molbev.a040023}.

\bibitem{csardi2006}
Csardi G, Nepusz T.
\newblock The igraph software package for complex network research.
\newblock InterJournal. 2006;Complex Systems:1695.
\newblock Available from: \url{https://igraph.org}.

\bibitem{stumpf2005}
Stumpf MPH, Wiuf C, May RM.
\newblock Subnets of scale-free networks are not scale-free: Sampling
  properties of networks.
\newblock Proceedings of the National Academy of Sciences.
  2005;102(12):4221--4224.
\newblock Available from: \url{https://www.pnas.org/content/102/12/4221}.

\end{thebibliography}

\end{document}